\documentclass[final]{siamltex}

\usepackage{epsfig}
\usepackage{amsmath,amssymb}
\usepackage{graphicx}
\usepackage{epstopdf}
\usepackage{amsmath,amssymb}

\newtheorem{rem}{Remark}[section]
\newtheorem{algorithm}{Algorithm}

\title{Inexact Krylov Subspace Algorithms for Large Matrix Exponential Eigenproblem from Dimensionality Reduction
}

\author{Gang Wu\thanks{Corresponding author. G. Wu is with Department of Mathematics, China University of Mining and
Technology \& School of Mathematics and Statistics, Jiangsu Normal University, Xuzhou, 221116, Jiangsu, P.R. China. E-mail: {\tt gangwu76@126.com} and {\tt
wugangzy@gmail.com}. This author is
supported by the National Science Foundation of China under grant
11371176, the Natural Science Foundation of Jiangsu Province under
grant BK20131126, the 333 Project of Jiangsu Province, and the
Talent Introduction Program of China University of Mining and
Technology.}
\and Ting-ting Feng\thanks{Department of Mathematics, Shanghai Key Laboratory of Pure Mathematics and Mathematical Practice, East China Normal University, Dongchuan RD 500, Shanghai, 200241,  P.R. China.
E-mail: {\tt tofengtingting@163.com}.}
\and Li-jia Zhang\thanks{School of Mathematics and statistics, Jiangsu Normal University, Xuzhou, 221116, Jiangsu, P.R.
China. E-mail: {\tt zhanglijia86@163.com}.}
\and Meng Yang\thanks{School of Computer Science and Technology, Soochow University, Suzhou, 215006, Jiangsu, P.R. China. E-mail: {\tt eyangmeng@163.com}.}
}

\begin{document}
\maketitle

\begin{abstract}
Matrix exponential discriminant analysis (EDA) is a generalized discriminant analysis
method based on matrix exponential. It can essentially
overcome the intrinsic difficulty of small sample size problem that exists in the classical
linear discriminant analysis (LDA). However, for data with high dimension, one has to solve a large matrix exponential eigenproblem in this method,
and the time complexity is dominated by
the computation of exponential of large matrices. In this paper, we propose two inexact Krylov subspace algorithms for solving
the large matrix exponential eigenproblem effectively. The contribution of this work is threefold. First, we consider how to compute matrix exponential-vector products efficiently, which is the key step in the Krylov subspace method. Second, we compare the discriminant analysis criterion of EDA and that of LDA from a theoretical point of view. Third, we establish a relationship between the accuracy of the approximate eigenvectors and the distance to nearest neighbour classifier, and show why the matrix exponential eigenproblem can be solved approximately in practice.
Numerical experiments on some real-world databases show superiority of our new algorithms over many state-of-the-art algorithms for face recognition.
\end{abstract}



\pagestyle{myheadings}
\thispagestyle{plain}
\pagestyle{myheadings} \thispagestyle{plain} \markboth{G. WU, T. FENG, L. ZHANG AND M. YANG}{INEXACT KRYLOV ALGORITHMS FOR MATRIX EXPONENTIAL EIGENPROBLEM}
{\section{Introduction}\label{sec1}}
\setcounter{equation}{0}

Face recognition has become one of the most successful applications of image analysis.
Real data for face images are
usually depicted in high dimensions.
In order to handle
high dimensional data, their dimensionality needs to be
reduced. In essence,
dimensionality reduction is the transformation of
high-dimensional data into a lower dimensional data space.
Currently, one of the most extensively used dimensionality reduction
methods is subspace transformation \cite{DHS,Wang,XWang,WXL}.

Linear
discriminant analysis (LDA) is one of notable subspace transformation
methods for dimensionality reduction \cite{DHS,RAF,PP}.
LDA encodes discriminant
information by maximizing the between-class scatter,
and meanwhile minimizing the within-class scatter in
the projected subspace. More precisely, let $\mathcal{X}=[\chi_{1},\chi_{2},\ldots,\chi_{n}]$ be a set of training samples in a $d$-dimensional feature space, and
assume that
the original data in $\mathcal{X}$ is partitioned into $k$ classes as
$\mathcal{X}=[\mathcal{X}_1,\mathcal{X}_2,\ldots,\mathcal{X}_k]$.
We denote by $n_{j}$ the number of samples in the $j$-th class, and thus $\sum_{j=1}^{k}n_{j}=n$. Let ${\bf \mu}_{j}$ be the centroid of the $j$-th class, and ${\bf \mu}$ be the global centroid of the training data set.
If we denote ${\bf e}_{j}=[1,1,\ldots,1]^{T} \in \mathbb{R}^{n_{j}}$, then the within-class scatter matrix is defined as
\begin{equation}\label{eqn1}
S_{W}=\sum_{j=1}^{k}\sum_{\chi_{i}\in \mathcal{X}_{j}}(\chi_{i}-{\bf \mu}_{j})(\chi_{i}-{\bf \mu}_{j})^{T}=H_WH_W^T,
\end{equation}
where
$H_{W}=[\mathcal{X}_{1}-\mu_{1}\cdot{\bf e}_{1}^{T},\ldots,\mathcal{X}_{k}-\mu_{k}\cdot{\bf e}_{k}^{T}]\in\mathbb{R}^{d\times n}$.
The between-class scatter matrix is defined as
\begin{equation}\label{eqn2}
S_{B}=\sum_{j=1}^{k}n_{j}({\bf \mu}_{j}-{\bf \mu})({\bf \mu}_{j}-{\bf \mu})^{T}=H_BH_B^T,
\end{equation}
where
$H_B=[\sqrt{n_{1}}({\bf \mu}_{1}-{\bf \mu}),\sqrt{n_{2}}({\bf \mu}_{2}-{\bf \mu}),\ldots,\sqrt{n_{k}}({\bf \mu}_{k}-{\bf \mu})]\in\mathbb{R}^{d \times k}$.
The LDA method is realized by maximizing the between-class scatter distance while minimizing the within-class scatter distance, which involves
solving a ``Trace Ratio" problem \cite{YJia,Ngo,NXJZY,Semi} in the form of
\begin{equation}\label{eqn11}
\tau=\max_{V\in\mathcal {R}^{d\times t} \atop V^{T}V = I}\frac{{\rm tr}(V^TS_BV)}{{\rm tr}(V^TS_WV)},
\end{equation}
where tr$(\cdot)$ denotes the trace of a matrix, and $t\ll d$ is the dimension of projection subspace. However, this problem is seldom solved in
practice. It is generally considered too difficult to solve and is commonly replaced by
a simpler, but not equivalent, problem called ``Ratio Trace" problem of the following form \cite{Ngo,PP}
\begin{equation} \label{1.2}
\begin{split}
\varrho=\max_{V\in\mathcal {R}^{d\times t} \atop V^{T}V = I}
{\rm tr}\big((V^{T}S_{W}V)^{-1}(V^{T}S_{B}V)\big),
 \end{split}
 \end{equation}
and the optimal projection matrix $V$ can be calculated from solving the following generalized symmetric eigenproblem
\begin{equation} \label{1.3}
\begin{split}
S_{B}{\bf x}= \lambda S_{W}{\bf x}.
 \end{split}
 \end{equation}

In practice, the dimension of real data usually exceeds the number of training
samples,
which results in the scatter matrix $S_W$ being singular. This is called the {\it small-sample-size} (SSS) or {\it undersampled
problem} \cite{PP}. It is an intrinsic limitation of the classical
LDA method, and it is also a common problem in classification
applications \cite{PP}.
In other words, the SSS problem stems from generalized
eigenproblems with singular matrices. To tackle the SSS problem,
many variants of LDA have been proposed in recent
years. To name a few, the regularized LDA method (RLDA) \cite{JHF}, LDA$+$PCA \cite{BHK}, the null-space LDA method (NLDA) \cite{CLKLY}, LDA/QR \cite{YL},
LDA/GSVD \cite{HJP,YJPP}, the direct LDA method (DLDA) \cite{YY}, the orthogonal LDA method (OLDA) \cite{Orth,Ye}, the neighborhood minmax projections (NMMP) \cite{minmax}, and so on. The above variations on LDA have both advantages and disadvantages \cite{PP,ZF}.
An {\it et al.} \cite{An} unified these LDA
variants in one framework: principal component analysis plus
constrained ridge regression.

Recently, a novel method based on matrix exponential, called exponential discriminant analysis (EDA), was proposed in \cite{ZF}. Instead of the LDA criterion (\ref{1.2}), EDA considers the following criterion
\begin{equation} \label{2.01}
\rho=\max_{\ V\in\mathcal {R}^{d\times t} \atop V^{T}V = I}
{\rm tr}\big((V^{T}{\rm exp}(S_{W})V)^{-1}(V^{T}{\rm exp}(S_{B})V)\big),
\end{equation}
where ${\rm exp}(\cdot)$ denotes exponential of a matrix or scalar in this paper. The projection matrix $V$ can be obtained from solving the $t$ dominant eigenvectors of the following {\it generalized symmetric matrix exponential eigenproblem} \cite{ZF}
\begin{equation} \label{2.02}
\textrm{exp}(S_{B}){\bf x}= \lambda \textrm{exp}(S_{W}){\bf x}.
\end{equation}

The framework of the EDA method for dimensionality reduction has gained wide attention in recent years.
For instance, an exponential locality preserving projections (ELPP) \cite{WCP} was proposed to
avoid the SSS problem occured in locality preserving projections (LPP). Wang {\it et al.} \cite{Wang} applied matrix exponential to extend many
popular Laplacian embedding algorithms such as locality preserving
projections, unsupervised discriminant projections, and marginal
fisher analysis.
A matrix exponential local discriminant embedding method (ELDE) was investigated in \cite{DB} to deal with the SSS problem appeared in local discriminant embedding (LDE).
Using the method of exponential discriminant analysis, Ahmed \cite{Ah} proposed a novel image clustering model that incorporated both local and global information in image database.
A 2DEDA method was presented in \cite{YPan}, which is an algorithm based on image matrices (2D data)
rather than image vectors (1D data). Thus, 2DEDA can be viewed as a generalization of the
EDA method to 2D data.

It has been widely shown that the EDA method has more discriminant power
than its original counterpart \cite{Ah,DB,WCP,Wang,YPan,ZF}. In the EDA framework, the matrix exponential can
be roughly interpreted as a random walk over the feature
similarity matrix, and thus is more robust. As the exponentials of $S_W$ and $S_B$ are symmetric positive definite (SPD),
the EDA method naturally deals with the SSS problem. Moreover, the behavior of the decay property of matrix exponential
is more significant in emphasizing small distance pairs \cite{Wang}.
However, in all the EDA-based methods, one has to solve the large matrix exponential eigenproblem (\ref{2.02}) for data with high dimensionality \cite{Ah,DB,WCP,Wang,YPan,ZF},
and the time complexity is dominated by
the computation of ${\rm exp}(S_B)$ and ${\rm exp}(S_W)$ \cite[pp.191]{ZF}.
This cost will be prohibitively large as the exponential of a matrix is often dense, even if the matrix in question is sparse \cite{Higham}.
Thus, for data with high dimension, the EDA method often suffers from heavy overhead and storage requirement.
So it is urgent to seek new technologies to speed up the solution of the large generalized matrix exponential eigenproblem arising in the framework of EDA.

Modern numerical linear algebra exploits the Krylov subspace method in different
advanced iterative procedures for large scale eigenvalue problems \cite{Bai,Saad,Stewart}. Indeed, this type of method ranks among ``{\it The Top 10
Algorithms of the 20th Century}" \cite{Top10}. In this paper, we devote ourselves to solving (\ref{2.02}) with the Krylov subspace method.
The key involves the evaluation of matrix exponential-vector products, which is a hot topic in large scientific and engineering computations \cite{Higham,ML2}. In conventional approaches, one has to evaluate these products by using some iterative methods \cite{Higham,ML2}.
In this work, we derive closed-form formulae for the matrix exponential-vector products, so that these products can be formed very efficiently. The second contribution of this work is to give a theoretical comparison for the discriminant analysis criteria (\ref{1.2}) and (\ref{2.01}) of LDA and EDA, respectively, and show the reason why EDA can improve the classification rate of the original LDA method. Finally, we establish a relationship between the accuracy of the approximate eigenvectors and the distance to the nearest neighbour classifier (NN) \cite{near}, and shed light on why the matrix exponential eigenproblem can be solved approximately in practical calculations.

This paper is organized as follows. In Section 2, we briefly overview the EDA method. In Section 3, we propose two inexact Krylov subspace algorithms for EDA. Numerical experiments on some real-world face recognition sets including AR, CMU-PIE, Extended YaleB, FERET, ORL, and Yale are performed in Section 4. Concluding remarks are given in Section 5.
Some notations used are listed in Table 1.1.

{\small
\begin{table*}[!t]
\caption{{\rm Notations used in this paper}} \vspace*{-12pt}
\begin{center}
\def\temptablewidth{1\textwidth}
{\rule{\temptablewidth}{1pt}}
\begin{tabular*}{\temptablewidth}{@{\extracolsep{\fill}}lccr}
{\bf Notations}     &{\bf Descriptions}  \\\hline
$d$ & Data dimension \\
$k$ &Number of classes\\
$n$ &Number of samples\\
$t$ &Dimension of projection subspace (or number of desired eigenpairs)\\
$\ell$ & Number of samples in the training set\\
$\mathcal{X}$     &Training samples in a $d$-dimensional feature space  \\
$\mathcal{X}_i$   &The $i$-th class of training samples   \\
$\chi_i$  &The $i$-th sample\\
$S_W$             &The within-class scatter matrix defined in (\ref{eqn1})    \\
$S_B$    &The between-class scatter matrix defined in (\ref{eqn2})  \\
${\rm exp}(S_W),{\rm exp}(S_B)$    & Exponentials of $S_W$ and $S_B$  \\
${\rm exp}^{1/2}(-S_W)$    &Square root of ${\rm exp}(-S_W)$ \\
$\lambda,{\bf x}$ &Eigenvalue and eigenvector\\
$\mathcal{K}_{m}(A,{\bf v}_{1})$ & Krylov subspace with respect to $A$ and ${\bf v}_1$\\
${\rm tr}(A)$    &Trace of the matrix $A$ \\
$V$    &Projection matrix composed of the desired eigenvectors  \\
$A^{T}$ & Transpose of $A$\\
$I$ &The identity matrix\\
$O$ &Zero matrix or vector\\
span$\{W\}$ & Subspace spanned by the columns of $W$\\
dim$(\mathcal{W})$ & Dimension of the subspace $\mathcal{W}$ \\
$\|\cdot\|_2$ & 2-norm of a vector or matrix
 \end{tabular*}
 {\rule{\temptablewidth}{1pt}}\\
 \end{center}
 \end{table*}
}

\section{The exponential discriminant analysis method}\label{sec2}
\setcounter{equation}{0}
Given an arbitrary square matrix $A$, its exponential is defined as \cite{GV,Higham}
\begin{equation}
{\rm exp}(A)=\sum_{j=0}^{\infty}\frac{A^{j}}{j!}=I+A+\frac{A^2}{2}+\cdots+\frac{A^s}{s!}+\cdots,
\end{equation}
where $I$ is the identity matrix. An important consequence is that {\it any} matrix exponential is {\it invertible}. Indeed, we have
$$
{\rm exp}(-A)\cdot{\rm exp}(A)={\rm exp}(A-A)=I,
$$
i.e.,
\begin{equation}\label{2.2}
{\rm exp}^{-1}(A)={\rm exp}(-A).
\end{equation}
Suppose that $A\in\mathbb{R}^{d\times d}$ is symmetric and let $A=X\Lambda X^{-1}$ be the eigen-decomposition of $A$, where $\Lambda={\rm diag}\{\lambda_1,\lambda_2,\ldots,\lambda_d\}$ is diagonal and $X$ is a $d\times d$ matrix with its columns being the eigenvectors. Then it is easy to see that
\begin{equation}
{\rm exp}(A)=X{\rm exp}(\Lambda)X^{-1}.
\end{equation}
In other words, $A$ and ${\rm exp}(A)$ share the same eigenvectors, and the eigenvalues of ${\rm exp}(A)$ are ${\rm exp}(\Lambda)={\rm diag}\{{\rm exp}(\lambda_1),{\rm exp}(\lambda_2),\ldots,{\rm exp}(\lambda_d)\}$.

The exponential discriminant analysis method (EDA) makes use of the exponential criterion (\ref{2.01}), and the projection matrix $V$
is obtained from solving the matrix exponential eigenproblem (\ref{2.02}). This method is equivalent to transforming the original
data into a new space by distance diffusion mapping, and
the LDA criterion is applied in such a new space \cite{ZF}.
The EDA algorithm is described as follows:
\begin{algorithm} \label{Alg1} {\bf The exponential discriminant analysis method (EDA)}~{\rm\cite{ZF}}\\
\textbf{Input:} The data matrix $\mathcal{X}=[\chi_{1},\chi_{2},\ldots,\chi_{n}]\in\mathbb{R}^{d\times n}$, where $\chi_{j}$ represernts the $j$-th training image.\\
\textbf{Output:} The projection matrix $V$.\\
1.~Compute the matrices $S_{B}$, $S_{W}$, ${\rm exp}(S_{B})$, and ${\rm exp}(S_{W})$;\\
2.~Compute the eigenvectors $\{{\bf x}_{i}\}$ and
eigenvalues $\{\lambda_{i}\}$ of ${\rm exp}(S_{W})^{-1}{\rm exp}(S_{B})$;\\
3.~Sort the eigenvectors $V=\{{\bf x}_{i}\}$
 according to $\lambda_{i}$ in decreasing order;\\
4.~Orthogonalize the columns of the projection matrix $V$.
\end{algorithm}

As a result of diffusion mapping, the
margin between different classes is enlarged in the EDA framework, which is helpful in
improving classification accuracy \cite{ZF}.
In the small-sample-size
case, EDA can extract not only the most discriminant information
that is contained in the null space of the within-class
scatter matrix, where it is similar to NLDA; but also the
discriminant information that is contained in the non-null
space of the within-class scatter matrix, and it is equivalent
to LDA$+$PCA \cite{ZF}.
In addition,
there are at most $k-1$ nonzero generalized
eigenvalues in LDA, where $k$ is the number of classes in the
data set, so the dimension of the projected subspace
is at most $k-1$ in LDA. As a comparison, there is no (theoretical) dimensionality limitation
in EDA, and we can project the input data into a low-dimensional
subspace whose dimension is larger than the number of classes.

{\bf Example 1.1}
{\it We illustrate the superiority of the EDA method over the LDA method via a toy example.
Consider the set of training samples
$\mathcal{X}=[\chi_1,\chi_2,\chi_3]\in\mathbb{R}^{3\times 3}$, whose dimension is equal to the number of samples.
Let the original data be partitioned into $k=3$ classes:
$
\mathcal{X}=[\chi_1,\chi_2,\chi_3]=[\mathcal{X}_1,\mathcal{X}_2,\mathcal{X}_3],
$
and assume that the columns of $\mathcal{X}$ are linear independent. So we have $\mu_i=\chi_i,i=1,2,3$, the global centroid $\mu=(\mu_1+\mu_2+\mu_3)/3$,
and
$$
H_W=[\chi_1-\mu_1,\chi_2-\mu_2,\chi_3-\mu_3]=O,
$$
$$
H_B=[\chi_1-\mu,\chi_2-\mu,\chi_3-\mu].
$$
Therefore, $S_W=H_WH_W^T=O$, $S_B=H_BH_B^T$, and the generalized eigenproblem {\rm(}\ref{1.3}{\rm)} for the LDA method becomes
\begin{equation}\label{eqn2.4}
S_B{\bf x}=O.
\end{equation}
Since the columns of $\mathcal{X}$ are linear independent, we have ${\rm rank}(S_B)=2$, and there is only one solution vector {\rm(}up to a scalar{\rm)} for {\rm(}\ref{eqn2.4}{\rm)}.

On the other hand, the matrix-exponential eigenproblem {\rm(}\ref{2.02}{\rm)} for the EDA method reduces to
\begin{equation}
\exp{(S_B)}{\bf x}=\lambda{\bf x},
\end{equation}
which has three linear independent solution vectors. Therefore, the LDA method loses some useful discriminant information, and the EDA method can perform better than the LDA method.}

As was stressed in \cite[pp.191]{ZF}, however, the time complexity of EDA is dominated by
the computation of ${\rm exp}(S_B)$ and ${\rm exp}(S_W)$, as well as the evaluation of the large matrix exponential eigenproblem. More precisely,
one has to explicitly form and store ${\rm exp}(S_B)$ and
${\rm exp}(-S_W)$, as well as their product ${\rm exp}(-S_{W}){\rm exp}(S_{B})$. The
complexity is $\mathcal{O}(d^3)$, which is prohibitively large as the dimension of the data is high \cite{Higham}. Moreover, we have to solve a large matrix exponential eigenproblem with respect to ${\rm exp}(S_{W})^{-1}{\rm exp}(S_{B})$. If the matrix exponential eigenproblem is solved by using the implicit QR algorithm \cite{Bai,GV}, another $\mathcal{O}(d^3)$ flops are required \cite{Bai,GV}.
Therefore, it is necessary to seek new strategies to improve the numerical performance of the EDA method.

\section{Inexact Krylov subspace algorithms for matrix exponential discriminant analysis} \label{sec3}
\setcounter{equation}{0}
In this section, we propose two inexact Krylov subspace algorithms for solving (\ref{2.02}).
We first consider how to compute the matrix exponential-vector products involved in the Krylov subspace method, and then
derive new lower and upper bounds on the EDA criterion (\ref{2.01}) and the LDA  criterion (\ref{1.2}).
Finally, we provide a practical stopping criterion for solving the matrix-exponential eigenproblem approximately.

\subsection{Computing the matrix exponential-vector products efficiently}

The generalized matrix exponential eigenproblem (\ref{2.02}) is mathematically equivalent to the following {\it standard nonsymmetric matrix exponential eigenproblem} \begin{equation} \label{26}
\textrm{exp}(-S_{W})\textrm{exp}(S_{B}){\bf x}= \lambda{\bf x}, \\
\end{equation}
where we used $\textrm{exp}^{-1}(S_{W})=\textrm{exp}(-S_{W})$.
As it is only required to calculate $t~(t\ll d)$ dominant eigenpairs of the large matrix $\textrm{exp}(-S_{W})\textrm{exp}(S_{B})$, we are interested in solving the large matrix exponential eigenproblem by using the Krylov subspace method.

The Arnoldi method is a widely used Krylov subspace method for finding a few extreme eigenvalues
and their associated eigenvectors of a large nonsymmetric matrix \cite{Bai,Saad,Stewart}. It requires only matrix-vector products to extract eigen-information
to compute the desired solutions, and thus is very attractive in practice when the matrix is
too large to be solved by, e.g., using the implicit QR algorithm \cite{Bai,GV}; or the matrix
does not exist explicitly but in the form of being capable of generating matrix-vector
multiplications.

The principle of the Arnoldi method is described as follows.
Given a real nonsymmetric matrix $A$ and an initial real vector ${\bf v}_1$ of unit norm, in exact arithmetics, the $m$-step Arnoldi process recursively generates an orthonormal basis
$\mathcal{V}_m=[{\bf v}_1,{\bf v}_2,\ldots,{\bf v}_m]$ for the Krylov subspace
$$
\mathcal{K}_m(A,{\bf v}_1)={\rm span}\{{\bf v}_1, A{\bf v}_1,\ldots,A^{m-1}{\bf v}_1\}.
$$
At the same time, the
projection of $A$ onto $\mathcal{K}_m(A,{\bf v}_1)$ is expressed as an upper Hessenberg matrix $H_m=\mathcal{V}_m^T A\mathcal{V}_m$. Afterwards the
Rayleigh-Ritz procedure \cite{Bai,Saad,Stewart} is applied to compute approximate eigenpairs of $A$: Let $(\widetilde{\lambda},\widetilde{\bf y})$ be some eigenpairs
of $H_m$, then we construct some approximate eigenpairs $(\widetilde{\lambda},\mathcal{V}_m\widetilde{\bf y})$ of $A$. Here $\widetilde{\lambda}$ is called a Ritz value and $\mathcal{V}_m\widetilde{\bf y}$ is called a Ritz vector.

Thus, denote $A={\rm exp}(-S_W){\rm exp}(S_B)$, the key ingredient of the Arnoldi method is to evaluate the matrix exponential-vector products
\begin{equation}\label{319}
{\bf w}=A{\bf v}={\rm exp}(-S_W){\rm exp}(S_B){\bf v}
\end{equation}
for different given vectors ${\bf v}$.
When the matrices $S_W,S_B$ are large, Krylov subspace methods are widely used for this type of problem \cite{Higham,ML2}.
Generally speaking, there are two classes of Krylov subspace methods for evaluating (\ref{319}) {\it iteratively} \cite{PS2}. In the first class of methods, the matrix is projected into
a much smaller subspace, then the exponential is applied to the reduced matrix,
and finally the approximation is projected back to the original large space \cite{Saad1}.
In the second class of methods, the exponential function is first approximated
by simpler functions such as rational functions,
and then the action of the matrix exponential is
calculated \cite{WW}. In this case, the core of (\ref{319}) reduces to solving some linear
systems with multiple shifts \cite{WPS}.

Thanks to the structure of ${\rm exp}(-S_W)$ and ${\rm exp}(S_B)$, instead of evaluating (\ref{319}) {\it iteratively}, we compute it directly by using a {\it closed-form} formula. More precisely, let
\begin{equation}\label{3.10}
H_{W}=Q_{W}R_{W},\quad H_{B}=Q_{B}R_{B}
\end{equation}
be the (skinny) QR decomposition \cite{GV} of the matrices $H_{W}$ and $H_{B}$ from (\ref{eqn1}) and (\ref{eqn2}), respectively, where $Q_{B}\in\mathbb{R}^{d\times k}, Q_{W}\in\mathbb{R}
^{d\times n}$ are orthonormal, and $R_{B}\in\mathbb{R}^{k\times k},
R_{W}\in\mathbb{R}^{n\times n}$ are upper triangular. Denote by
\begin{equation}\label{3.11}
R_{B}=U_{B}\Sigma_{B}V_{B}^{T},\quad
R_{W}=U_{W}\Sigma_{W}V_{W}^{T}
\end{equation}
the singular value decomposition (SVD) \cite{GV} of $R_B$ and $R_W$, respectively, and let
\begin{equation}\label{3.12}
\widetilde{Q}_{B}=Q_{B}U_{B}, \quad D_{B}=\Sigma_{B}^{2},
\end{equation}
and
\begin{equation}\label{3.13}
\widetilde{Q}_{W}=Q_{W}U_{W},\quad
D_{W}=\Sigma_{W}^{2},
\end{equation}
Then we have the following theorem:
\begin{theorem}\label{Thm3.1}
Let $\widetilde{Q}_B^{\bot}$ and $\widetilde{Q}_W^{\bot}$ be orthonormal bases for the orthogonal complement of ${\rm span}\{\widetilde{Q}_B\}$ and ${\rm span}\{\widetilde{Q}_W\}$, respectively, and let $f$ be a function defined on the spectra of $S_W$ and $S_B$ defined in {\rm(}\ref{eqn1}{\rm)} and {\rm(}\ref{eqn2}{\rm)}, respectively. Then
\begin{equation}\label{eqn317}
f(S_B)=\widetilde{Q}_Bf(D_B)\widetilde{Q}_B^{T}+\widetilde{Q}_B^{\bot}f(O_{(d-k)\times (d-k)})(\widetilde{Q}_B^{\bot})^{T}
\end{equation}
and
\begin{equation}\label{eqn318}
f(S_W)=\widetilde{Q}_Wf(D_W)\widetilde{Q}_W^{T}+\widetilde{Q}_W^{\bot}f(O_{(d-n)\times (d-n)})(\widetilde{Q}_W^{\bot})^{T},
\end{equation}
where $O_{(d-k)\times (d-k)}$ and $O_{(d-n)\times (d-n)}$ are $(d-k)\times (d-k)$ and $(d-n)\times (d-n)$ zero matrices, respectively.
In particular, for the matrix exponential, we have that
\begin{equation}\label{317}
{\rm exp}(S_B)=\widetilde{Q}_{B}{\rm exp}(D_B)\widetilde{Q}_{B}^{T}
+ I - \widetilde{Q}_{B}\widetilde{Q}_{B}^{T},
\end{equation}
\begin{equation}\label{316}
{\rm exp}(-S_W)=
\widetilde{Q}_{W}{\rm exp}(-D_W)\widetilde{Q}_{W}^{T} + I -
\widetilde{Q}_{W}\widetilde{Q}_{W}^{T},
\end{equation}
and
\begin{equation}\label{318}
{\rm exp}^{1/2}(-S_W)=\widetilde{Q}_{W}{\rm exp}\Big(-\frac{1}{2}D_{W}\Big)\widetilde{Q}_{W}^{T}+I -
\widetilde{Q}_{W}\widetilde{Q}_{W}^{T},
\end{equation}
where ${\rm exp}^{1/2}(-S_{W})$ represents the square root of ${\rm exp}(-S_{W})$.
\end{theorem}
\begin{proof}
We only prove (\ref{eqn317}), and the proof of (\ref{eqn318}) is similar. It follows from (\ref{eqn1}), (\ref{eqn2}) and (\ref{3.10})--(\ref{3.13}) that
\begin{eqnarray*}
S_{B}&=&H_{B}H_{B}^{T}=(Q_{B}U_{B})\Sigma_{B}(V_{B}^{T}V_{B})\Sigma_{B}(U_{B}^{T}Q_{B}^{T})=\widetilde{Q}_{B}D_{B}\widetilde{Q}_{B}^{T}\\
&=&\left[\begin{array}{cc}
\widetilde{Q}_B & \widetilde{Q}_B^{\bot}
\end{array} \right]
\left[\begin{array}{cc}
D_B &  \\  & O_{(d-k)\times (d-k)}
\\\end{array}\right]
\left[\begin{array}{c}
\widetilde{Q}_B^{T} \\ (\widetilde{Q}_B^{\bot})^{T}
\\\end{array}\right].
\end{eqnarray*}
From the properties of matrix function \cite[Theorem 1.13]{Higham}, we obtain
\begin{eqnarray*}
f(S_{B})&=&\left[\begin{array}{cc}
\widetilde{Q}_B & \widetilde{Q}_B^{\bot}
\end{array} \right]
\left[\begin{array}{cc}
f(D_B) &  \\  & f(O_{(d-k)\times (d-k)})
\\\end{array}\right]
\left[\begin{array}{c}
\widetilde{Q}_B^{T} \\ (\widetilde{Q}_B^{\bot})^{T}
\\\end{array}\right]\\
&=&\left[\begin{array}{cc}
\widetilde{Q}_B f(D_B) & \widetilde{Q}_B^{\bot}f(O_{(d-k)\times (d-k)})
\end{array} \right]
\left[\begin{array}{c}
\widetilde{Q}_B^{T} \\ (\widetilde{Q}_B^{\bot})^{T}
\\\end{array}\right]\\
&=&\widetilde{Q}_Bf(D_B)\widetilde{Q}_B^{T}+\widetilde{Q}_B^{\bot}f(O_{(d-k)\times (d-k)})(\widetilde{Q}_B^{\bot})^{T}.
\end{eqnarray*}
And (\ref{317})--(\ref{318}) follow from
$$
\widetilde{Q}_B^{\bot}(\widetilde{Q}_B^{\bot})^{T}=I-\widetilde{Q}_B\widetilde{Q}_B^{T},~\widetilde{Q}_W^{\bot}(\widetilde{Q}_W^{\bot})^{T}
=I-\widetilde{Q}_W\widetilde{Q}_W^{T},
$$
and the fact that $\exp(O)=I$ for any square zero matrix.
\end{proof}


\begin{rem}
According to {\rm(}\ref{317}{\rm)} and {\rm(}\ref{316}{\rm)}, for a given vector ${\bf v}$, we can compute the matrix exponential-vector product {\rm(}\ref{319}{\rm)}
within two steps:
\begin{flushleft}
{\rm(i)}~${\bf v}=\widetilde{Q}_{B}{\rm exp}(D_B)(\widetilde{Q}_{B}^{T}{\bf v})
+ {\bf v} - \widetilde{Q}_{B}(\widetilde{Q}_{B}^{T}{\bf v})$;\\
{\rm(ii)}~${\bf v}=\widetilde{Q}_{W}{\rm exp}(-D_W)(\widetilde{Q}_{W}^{T}{\bf v})
+{\bf v} - \widetilde{Q}_{W}(\widetilde{Q}_{W}^{T}{\bf v})$.
\end{flushleft}
Notice that $D_B$ and $D_W$ are $k\times k$ and $n\times n$ diagonal matrices, respectively.
As a result, there is no need to explicitly form and store the large matrices $S_B$, $S_W$ and their exponentials ${\rm exp}(S_{B})$, ${\rm exp}(-S_{W})$ in the EDA framework.
\end{rem}


Indeed, the {\it generalized matrix exponential eigenproblem}
(\ref{2.02}) can also be reformulated as a {\it standard symmetric eigenvalue problem}, say, by using
${\rm exp}(S_W)$'s Cholesky decomposition \cite{GV}. Unfortunately, calculating the Cholesky decomposition of a large matrix can be very expensive. Rather
than using the Cholesky decomposition, we consider ${\rm exp}(S_W)$'s square root \cite{Higham} that is advantageous
for both computational and theoretical considerations.
By (\ref{2.02}),
\begin{eqnarray}\label{3.1}
\big[{\rm exp}^{-1/2}(S_{W}){\rm exp}(S_B){\rm exp}^{-1/2}(S_{W})\big]\big[{\rm exp}^{1/2}(S_{W}){\bf x}\big]
=\lambda \big[{\rm exp}^{1/2}(S_{W}){\bf x}\big].
\end{eqnarray}
Denote ${\rm exp}^{-1/2}(S_{W})$ the inverse of square root of ${\rm exp}(S_{W})$, and note that ${\rm exp}^{-1/2}(S_{W})={\rm exp}^{1/2}(-S_{W})$, (\ref{3.1}) can be rewritten as
\begin{equation}\label{3.2}
\big[{\rm exp}^{1/2}(-S_{W}){\rm exp}(S_B){\rm exp}^{1/2}(-S_{W})\big]{\bf y}=\lambda{\bf y},
\end{equation}
and
\begin{equation}\label{eqn311}
{\bf x}={\rm exp}^{1/2}(-S_{W}){\bf y}
\end{equation}
is the desired solution.
Notice that the matrix
\begin{equation}\label{eqn315}
M={\rm exp}^{1/2}(-S_{W}){\rm exp}(S_B){\rm exp}^{1/2}(-S_{W})
\end{equation}
is a symmetric positive definite (SPD) matrix, and our aim is to compute a few dominate eigenpairs of it. The (symmetric) Lanczos method \cite{Bai,Saad,Stewart} is a widely used Krylov subspace method for finding a few extreme eigenvalues
and their associated eigenvectors of a large symmetric matrix.
The Lanczos method can be viewed as a simplification of Arnoldi's method
for the particular case where the matrix in question is symmetric. In this case the Hessenberg
matrix $H_m$ becomes symmetric tridiagonal, which leads to a three-term recurrence in the Lanczos process.
Similar to the Arnoldi process, in the Lanczos process, we need to compute the matrix exponential-vector products
\begin{equation}\label{320}
{\bf w}=M{\bf v}={\rm exp}^{1/2}(-S_{W}){\rm exp}(S_B){\rm exp}^{1/2}(-S_{W}){\bf v}
\end{equation}
with different vectors ${\bf v}$. In terms of (\ref{317}), (\ref{316}) and (\ref{318}), this problem can be solved within three steps:
\begin{flushleft}
(iii)~${\bf v}=\widetilde{Q}_{W}{\rm exp}(-1/2\times D_{W})\widetilde{Q}_{W}^{T}{\bf v}+{\bf v}-
\widetilde{Q}_{W}\widetilde{Q}_{W}^{T}{\bf v}$;\\
(iv)~${\bf v}=\widetilde{Q}_{B}{\rm exp}(D_B)\widetilde{Q}_{B}^{T}{\bf v}
+ {\bf v} - \widetilde{Q}_{B}\widetilde{Q}_{B}^{T}{\bf v}$;\\
(v)~${\bf v}=\widetilde{Q}_{W}{\rm exp}(-1/2\times D_{W})\widetilde{Q}_{W}^{T}{\bf v}+{\bf v}-
\widetilde{Q}_{W}\widetilde{Q}_{W}^{T}{\bf v}$.
\end{flushleft}

Compared with the Arnoldi method for the standard nonsymmetric matrix exponential eigenproblem {\rm(}\ref{26}{\rm)}, the advantage of the Lanczos method
for {\rm(}\ref{3.2}{\rm)} is that one can use a three-term recurrence in the Lanczos procedure for a symmetric
eigenproblem \cite{Bai,Saad,Stewart}. The disadvantage is that one has to evaluate three matrix exponential-vector products {\rm(iii)--(v)} in each step of the Lanczos procedure, rather than two matrix exponential-vector products {\rm(i)--(ii)} in each step of the Arnoldi procedure.
So a natural question is: which one is better, Arnoldi or Lanczos? We believe that the answer is problem-dependent. Indeed, we cannot tell which one is {\it definitely better} than the other; refer to the numerical experiments given in Section 4.
The pseudo-code for computing (\ref{319}) and (\ref{320}) is listed as follows:

\begin{algorithm}\label{Alg2}~{\bf Computing the matrix exponential-vector product (\ref{319})/(\ref{320}) in the Arnoldi/Lanczos process}

{$\bullet$ \bf Preprocessing:}

1.~Given the training set $\mathcal{X}$, form the matrices $H_{B}$ and $H_{W}$;

2.~Computing the skinny QR decompositions: $H_{B}=Q_{B}R_{B}$, $H_{W}=Q_{W}R_{W}$;

3.~Computing SVDs: $R_{B}=U_{B}\Sigma_{B}V_{B}^{T}, R_{W}=U_{W}\Sigma_{W}V_{W}^{T}$;

4.~Let $Q_{B}=Q_{B}U_{B}$, $Q_{W}=Q_{W}U_{W}$, $\Sigma_{B}={\rm exp}(\Sigma_{B}^{2})$;

{$\bullet$ \bf Computing:}

5. Given a vector ${\bf v}$;

{\rm if computing (\ref{319}) in the Arnoldi process}
\begin{tabbing}
in \= in \= in \= in \= in \= in \= in \= \hspace{3.0truein} \= \kill
\>\>\>6A.~Let $\Sigma_{W}={\rm exp}(-\Sigma_{W}^{2})$;\\

\>\>\>7A.~${\bf w}=Q_{B}^{T}{\bf v}$;\\

\>\>\>8A.~${\bf v}=Q_{B}\Sigma_{B}{\bf w}+{\bf v}-Q_{B}{\bf w}$;\\

\>\>\>9A.~${\bf w}=Q_{W}^{T}{\bf v}$;\\

\>\>\>10A.~${\bf w}=Q_{W}\Sigma_{W}{\bf w}+{\bf v}-Q_{W}{\bf w}$;\\

\>~{\rm end}
\end{tabbing}

{\rm if computing (\ref{320}) in the Lanczos process}
\begin{tabbing}
in \= in \= in \= in \= in \= in \= in \= \hspace{3.0truein} \= \kill

\>\>\>6L.~Let $\Sigma_{W}={\rm exp}(-1/2\times\Sigma_{W}^{2})$;\\

\>\>\>7L.~${\bf w}=Q_{W}^{T}{\bf v}$;\\

\>\>\>8L.~${\bf v}=Q_{W}\Sigma_{W}{\bf w}+{\bf v}-Q_{W}{\bf w}$;\\

\>\>\>9L.~${\bf w}=Q_{B}^{T}{\bf v}$;\\

\>\>\>10L.~${\bf v}=Q_{B}\Sigma_{B}{\bf w}+{\bf v}-Q_{B}{\bf w}$;\\

\>\>\>11L.~${\bf w}=Q_{W}^{T}{\bf v}$;\\

\>\>\>12L.~${\bf w}=Q_{W}\Sigma_{W}{\bf w}+{\bf v}-Q_{W}{\bf w}$;\\

\>~{\rm end}
\end{tabbing}

\end{algorithm}
\begin{rem}
In practical implementations, the ``Preprocessing" phase 1--4 is run once for all, and the variables $Q_B,Q_W,\Sigma_B$ and $\Sigma_W$ can be stored for a latter use.
In this procedure, one needs to perform two skinny QR decompositions in $\mathcal{O}(dk^2)$ and $\mathcal{O}(dn^2)$ flops \cite{GV}, respectively. Further, it is required to perform SVD of two small-sized matrices $R_B$ and $R_W$, in $\mathcal{O}(k^3)$ and $\mathcal{O}(n^3)$ flops \cite{GV}, respectively; as well as to compute a small-sized diagonal matrix exponential in $\mathcal{O}(k)$ flops \cite{Higham}, which are negligible as $k\leq n\ll d$.
In the ``Computing" phase 6A--10A and 6L--12L, the computational cost is $\mathcal{O}\big((k+n)d\big)$. Thus, as $k\leq n\ll d$, the main overhead for solving {\rm(}\ref{319}{\rm)} and {\rm(}\ref{320}{\rm)} is $\mathcal{O}(d)$.

On the other hand, if the $n$ training vectors are linear independent, the ranks of the matrices $S_{B}$ and $S_{W}$ are $k-1$ and $n-k$, respectively, and the main storage requirement of Algorithm 2 is to store $n-1$ vectors of length $d$. Therefore, the computational complexities of our new algorithms are much fewer than those of Algorithm \ref{Alg1}.
\end{rem}

\subsection{A comparison of the discriminant analysis criteria}

In this subsection, we will establish new lower and upper bounds for the criteria (\ref{1.2}) and (\ref{2.01}), and give a theoretical comparison of the LDA method and the EDA method. Recall from (\ref{26}) and (\ref{3.1}) that, the eigenvalues of the matrix $M$ from (\ref{eqn315}) are the same as those of ${\rm exp}(-S_{W}){\rm exp}(S_B)$ and the matrix exponential pencil $\big({\rm exp}(S_B),{\rm exp}(S_W)\big)$.

\begin{theorem}\label{Thm3.3}
Denote by $\nu_1\geq\nu_2\geq\cdots\geq\nu_d$ the eigenvalues of $S_W$, by $\mu_1\geq\mu_2\geq\cdots\geq\mu_d$ the eigenvalues of $S_B$, and let $\lambda_1(M)\geq\lambda_2(M)\geq\cdots\geq\lambda_d(M)$ be the eigenvalues of $M$. Then we have that for  $i=1,2,\ldots,d$,
\begin{equation}\label{333}
\lambda_{i}(M)\geq\max\big\{{\rm exp}(\mu_i-\nu_1),{\rm exp}(\mu_d-\nu_{d-i+1})\big\},
\end{equation}
and
\begin{equation}\label{444}
\lambda_{i}(M)
\leq\min\big\{{\rm exp}(\mu_i-\nu_d),{\rm exp}(\mu_1-\nu_{d-i+1})\big\}.
\end{equation}
\end{theorem}
\begin{proof}
Note that $\lambda_i(M)=\lambda_{i}\big({\rm exp}(-S_{W}){\rm exp}(S_B)\big),~i=1,2,\ldots,d$. Denote
$\widetilde{W}=\textrm{exp}^{1/2}(-S_{W})$, we have from the Courant-Fischer minimax theorem that \cite{GV}
\begin{equation} \label{3.3}
\begin{array}{lll}
&\lambda_{i}\big({\rm exp}(-S_{W}){\rm exp}(S_B)\big)=\lambda_{i}(\widetilde{W}^{T}{\rm exp}(S_{B})\widetilde{W})&\\
=&\max\limits_{{\rm dim}(S)=i}
~\min\limits_{
{\bf v}\in S,~\|{\bf v}\|_{2}=1}
{\bf v}^{T}\widetilde{W}^{T}{\rm exp}(S_{B})\widetilde{W}{\bf v}&\\
=&\max\limits_{{\rm dim}(S)=i}~\min\limits_{{\bf v}\in S,~\|{\bf v}\|_{2}=1} \frac{(\widetilde{W}{\bf v})^{T}\textrm{exp}(S_{B})(\widetilde{W}{\bf v})}{(\widetilde{W}{\bf v})^{T}(\widetilde{W}{\bf v})}\cdot\frac{{\bf v}^{T}\widetilde{W}^{T}\widetilde{W}{\bf v}}{{\bf v}^{T}{\bf v}}.&\\
\end{array}
\end{equation}
Recall that the eigenvalues of ${\rm exp}(-S_W)$ and ${\rm exp}(S_B)$ are ${\rm exp}(-\nu_d)\geq{\rm exp}(-\nu_{d-1})\geq\cdots\geq{\rm exp}(-\nu_1)$ and ${\rm exp}(\mu_1)\geq{\rm exp}(\mu_2)\geq\cdots\geq{\rm exp}(\mu_d)$, respectively. For any ${\bf v}\in \mathbb{R}^{d},~\|{\bf v}\|_2=1$, it follows that
\begin{equation}\label{3.4}
{\rm exp}(-\nu_1)\leq\frac{{\bf v}^{T}\widetilde{W}^{T}\widetilde{\
W}{\bf v}}{{\bf v}^{T}{\bf v}}=\frac{{\bf v}^{T}{\rm exp}(-S_W){\bf v}}{{\bf v}^{T}{\bf v}}\leq {\rm exp}(-\nu_d),
\end{equation}
and
\begin{equation}\label{3.5}
{\rm exp}(\mu_i)=\max_{{\rm dim}(S)=i}~\min_{{\bf v}\in S,~\|{\bf v}\|_{2}=1}\frac{(\widetilde{W}{\bf v})^{T}{\rm exp}(S_{B})(\widetilde{W}{\bf v})}{(\widetilde{W}{\bf v})^{T}(\widetilde{W}{\bf v})}.
\end{equation}
Therefore, we have from (\ref{3.3})--(\ref{3.5}) that
\begin{equation}\label{3.6}
{\rm exp}(\mu_i-\nu_1)\leq \lambda_{i}(M)
\leq {\rm exp}(\mu_i-\nu_d),\quad i=1,2,\ldots,d.
\end{equation}

On the other hand, we have
\begin{eqnarray*}
\lambda_{i}\big({\rm exp}(-S_{W}){\rm exp}(S_B)\big)=
\lambda_{i}\big({\rm exp}^{1/2}(S_B){\rm exp}(-S_{W}){\rm exp}^{1/2}(S_B)\big).
\end{eqnarray*}
Using the same trick, we can prove that for $i=1,2,\ldots,d$,
\begin{equation}\label{3.8}
{\rm exp}(\mu_d-\nu_{d-i+1})\leq \lambda_{i}(M)
\leq {\rm exp}(\mu_1-\nu_{d-i+1}).
\end{equation}
A combination of (\ref{3.6}) and (\ref{3.8}) yields (\ref{333}) and (\ref{444}).
\end{proof}

We are in a position to establish new lower and upper bounds for the $\rho$ value defined in (\ref{2.01}). First we need the following lemma.

\begin{lemma}\label{Lem3.4}\cite[Corollary 4.6.4, pp.101]{WWJ}
Let $A$ and $B$ be two $d\times d$ Hermitian matrix, where $A$ is semi-positive definite and $B$ is positive definte. Let
$X$ be a $d\times t$ matrix with ${\rm rank}(X)=t$. Then
$$
{\rm tr}\big(X^HAX(X^HBX)^{-1}\big)\geq\sum_{i=1}^t\lambda_{d-t+i}(B^{-1}A),
$$
and
$$
{\rm tr}\big(X^HAX(X^HBX)^{-1}\big)\leq\sum_{i=1}^t\lambda_{i}(B^{-1}A).
$$
\end{lemma}

Combining Theorem \ref{Thm3.3} and Lemma \ref{Lem3.4}, we get the theorem as follows for the exponential discriminant analysis criterion (\ref{2.01}).

\begin{theorem}\label{Thm3.5}
Under the above notations, for the EDA criterion, there holds
\begin{equation}\label{eqn320}
\rho\geq\sum_{i=1}^t\max\big\{{\rm exp}(\mu_{d-t+i}-\nu_1),{\rm exp}(\mu_d-\nu_{t-i+1})\big\},
\end{equation}
and
\begin{equation}\label{eqn3202}
\rho\leq\sum_{i=1}^t\min\big\{{\rm exp}(\mu_i-\nu_d),{\rm exp}(\mu_1-\nu_{d-i+1})\big\}.
\end{equation}
\end{theorem}
\begin{proof}
This result is from Lemma \ref{Lem3.4}, (\ref{333}), (\ref{444}), and the fact that for $i=1,2,\ldots,d$,
\begin{eqnarray*}
\lambda_{d-t+i}\big({\rm exp}(-S_W){\rm exp}(S_B)\big)\geq
\max\big\{{\rm exp}(\mu_{d-t+i}-\nu_1),{\rm exp}(\mu_d-\nu_{t-i+1})\big\}.
\end{eqnarray*}
\end{proof}

If $S_W$ is nonsingular, we have the following result on the LDA criterion (\ref{1.2}), whose proof is similar to that of Theorem \ref{Thm3.5}, and thus is omitted.
\begin{theorem}\label{Thm3.6}
If $S_W$ is nonsingular, let $\nu_1\geq\nu_2\geq\cdots\geq\nu_d>0$, $\mu_1\geq\mu_2\geq\cdots\geq\mu_d$ be the eigenvalues of $S_W$ and $S_B$, respectively, and let $\varrho$ be the LDA criterion defined in {\rm(}\ref{1.2}{\rm)}. Then we have that
\begin{equation}\label{eqn321}
\sum_{i=1}^t\max\left\{\frac{\mu_{d-t+i}}{\nu_1},\frac{\mu_d}{\nu_{t-i+1}}\right\}\leq\varrho\leq\sum_{i=1}^t\min\left\{\frac{\mu_i}{\nu_d},\frac{\mu_1}{\nu_{d-i+1}}\right\}.
\end{equation}
\end{theorem}

Without loss of generality, we assume that the $n$ training vectors are linear independent, so that the ranks of the matrices $S_{B}$ and $S_{W}$ are $k-1$ and $n-k$, respectively. As $k\leq n\ll d$, there are many 1's in the spectra of ${\rm exp}(-S_W)$ and ${\rm exp}(S_B)$.
The following theorem gives the number of 1 eigenvalues of $M$.
\begin{theorem}\label{Thm3.7}
If the $n$ samples in $\mathcal{X}$ are linear independent,
then $M$ has at least $d-n+1$ eigenvalues that are equal to 1.
\end{theorem}

\begin{proof}
Since the $n$ samples are linear independent, there are $d-k+1$ and $d-n+k$ zero eigenvalues in the spectra of $S_B$ and $S_W$, respectively. As a result, there are $d-k+1$ and $d-n+k$ eigenvalues of ${\rm exp}(S_B)$ and ${\rm exp}(S_W)$, respectively, which are equal to 1.
Let $\{{\bf u}_i\}_{i=1}^{d-k+1}$ and $\{{\bf g}_i\}_{i=1}^{d-n+k}$ be the orthnormal eigenvectors of ${\rm exp}(S_W)$ and ${\rm exp}(S_B)$ corresponding to the eigenvalue 1, respectively, and denote by
$$
\Phi_1={\rm span}\{{\bf u}_1,\ldots,{\bf u}_{d-k+1}\},\quad\Phi_2={\rm span}\{{\bf g}_1,\ldots,{\bf g}_{d-n+k}\}
$$
the corresponding eigenspace (or invariant subspace) of ${\rm exp}(S_B)$ and ${\rm exp}(S_W)$, respectively.
Let ${\rm dim}(\Phi_i)$ be the dimension of the space $\Phi_i~(i=1,2)$, from the the well-know dimension formula \cite[pp.2]{WWJ}
$$
d\geq{\rm dim}(\Phi_1+\Phi2)={\rm dim}(\Phi_1)+{\rm dim}(\Phi_2)-{\rm dim}(\Phi_1\cap\Phi_2),
$$
we obtain
$$
{\rm dim}(\Phi_1\cap\Phi_2)\geq d-n+1.
$$
Therefore, there are at least $d-n+1$ independent vectors $\{{\bf z}\}_{i=1}^{d-n+1}\in\Omega_1\cap\Omega_2$, such that
$$
{\rm exp}(S_W){\bf z}_i={\bf z}_i\quad{\rm and}\quad {\rm exp}(S_B){\bf z}_i={\bf z}_i.
$$
That is,
$$
{\rm exp}(-S_W){\rm exp}(S_B){\bf z}_i={\bf z}_i,\quad i=1,2,\ldots,d-n+1,
$$
and $M$ has at least $d-n+1$ eigenvalues that are equal to 1.
\end{proof}

\begin{rem}
It was assumed in \cite[pp.189]{ZF} that if
\begin{equation}\label{eqn322}
\frac{{\rm exp}(\mu_i)}{{\rm exp}(\nu_i)}>\frac{\mu_i}{\nu_i},\quad i=1,2,\ldots,t,
\end{equation}
then there is a difference in diffusion scale between the
within-class and between-class distances, and the diffusion scale to
the between-class distance is larger than the within-class distance.
Hence, the margin between different classes is enlarged, which is helpful in improving classification
accuracy.
Theorems \ref{Thm3.5}--\ref{Thm3.7} show this more precisely. It indicates that if 
\begin{equation}\label{eqn323}
\frac{{\rm exp}(\mu_i)}{{\rm exp}(\nu_d)}>\frac{\mu_i}{\nu_d}
\quad or\quad
\frac{{\rm exp}(\mu_1)}{{\rm exp}(\nu_{d-i+1})}>\frac{\mu_1}{\nu_{d-i+1}}, i=1,2,\ldots,t,
\end{equation}
then
$\rho$ can be larger than $\varrho$. Notice that the conditions in {\rm(}\ref{eqn323}{\rm)} are {\rm (}much{\rm)} weaker than those in {\rm(}\ref{eqn322}{\rm)}.
So our result is stronger than the one given in \cite{ZF}, and can be useful to clarify the motivation of EDA.
\end{rem}

\subsection{Solving the matrix exponential eigenproblems approximately}

In the new strategy, one needs to solve the large matrix exponential eigenproblems (\ref{26}) or (\ref{3.1}) in a prescribed accuracy. If the desired accuracy is high (say, $tol=10^{-8}$ or even the machine precision $\epsilon\approx 2.22\times 10^{-16}$), then the cost for solving the large eigenproblems will be very high. In this subsection, we give theoretical analysis on the relationship between the accuracy of the eigenvectors and the distance measure of the nearest neighbour classifier (NN) \cite{near}. Based on the theoretical results, we provide a practical stopping criterion for the matrix exponential eigenproblems.

K nearest neighbors (KNN) is a simple algorithm that stores all available cases and classifies new cases based on a similarity measure \cite{KNN}, e.g., the Euclidean distance. This algorithm has been used in statistical estimation and pattern recognition in the beginning of 1970's as a non-parametric technique.
In KNN classification, the output is a class membership. An object is classified by a majority vote of its neighbors, with the object being assigned to the class most common among its K nearest neighbors (where K is a positive integer, typically small).
If $\text{K}=1$, then the case is simply assigned to the class of its nearest neighbor (NN) \cite{near}.

Denote by $V,\widetilde{V}\in\mathbb{R}^{d\times\ell}$ the orthonormal matrices whose columns are the ``exact" and ``computed" solutions of {\rm(}\ref{26}{\rm)} or {\rm(}\ref{3.1}{\rm)}, respectively. Let $\widehat{\bf x}_{i}$ be a sample from the training set, and $\widehat{\bf y}_{j}$ be a sample from the testing set. Then the nearest neighbour classifier gives class membership via investigating the Euclidean distance \cite{near}
\begin{equation}
d_{ij}=\|VV^T(\widehat{\bf x}_{i}-\widehat{\bf y}_{j})\|_2=\|V^T(\widehat{\bf x}_{i}-\widehat{\bf y}_{j})\|_2,
\end{equation}
where $\|\cdot\|_2$ denotes 2-norm (or Euclidean norm) of a vector or matrix. The following theorem sheds light on why we can solve the matrix exponential eigenproblems approximately.

\begin{theorem}\label{Thm3.8}
Let $V,\widetilde{V}\in\mathbb{R}^{d\times t}$ be orthonormal matrices whose columns are the ``exact" and ``computed" solutions of {\rm(}\ref{26}{\rm)} or {\rm(}\ref{3.1}{\rm)}, respectively.
Denote by $d_{ij}=\|V^T(\widehat{\bf x}_{i}-\widehat{\bf y}_{j})\|_2$ and $\widetilde{d}_{ij}=\|\widetilde{V}^T(\widehat{\bf x}_{i}-\widehat{\bf y}_{j})\|_2$ the ``exact" and ``computed" Euclidean distances, respectively, and let $\sin\angle(V,\widetilde{V})=\|(I-VV^T)\widetilde{V}\|_2$ be the distance between the eigenspace ${\rm span}\{V\}$ and the approximate eigenspace ${\rm span}\{\widetilde{V}\}$. If $\|\widehat{\bf x}_i\|_2,\|\widehat{\bf y}_j\|_2=1$ and $\cos\angle(V,\widetilde{V})\neq 0$, then
\begin{equation}\label{eqn324}
\frac{\widetilde{d}_{ij}-2\sin\angle(V,\widetilde{V})}{\cos\angle(V,\widetilde{V})}\leq d_{ij}\leq\widetilde{d}_{ij}\cos\angle(V,\widetilde{V})+2\sin\angle(V,\widetilde{V}).
\end{equation}
\end{theorem}
\begin{proof}
Note that
$V=\widetilde{V}\widetilde{V}^{T}V+(I-\widetilde{V}\widetilde{V}^{T})V$,
and $\cos\angle(V,\widetilde{V})=\|V^{T}\widetilde{V}\|_{2}$ \cite{Stewart}.
Thus,
\begin{eqnarray*}
d_{ij}&=& \|V^{T}(\widehat{\bf x}_{i}-\widehat{\bf y}_{j})\|_{2}\\
&=& \|(V^{T}\widetilde{V})\widetilde{V}^{T}(\widehat{\bf x}_{i}-\widehat{\bf y}_{j})+V^{T}(I-\widetilde{V}\widetilde{V}^{T})(\widehat{\bf x}_{i}-\widehat{\bf y}_{j})\|_{2}\\
&\leq&\widetilde{d}_{ij}\|V^{T}\widetilde{V}\|_{2}+\|V^{T}(I-\widetilde{V}\widetilde{V}^{T})\|_{2}\cdot\|\widehat{\bf x}_{i}-\widehat{\bf y}_{j}\|_{2}\\
&=&\widetilde{d}_{ij}\cdot\cos\angle(V,\widetilde{V})+\sin\angle(V,\widetilde{V})\cdot\|\widehat{\bf x}_{i}-\widehat{\bf y}_{j}\|_{2}\\
&\leq&\widetilde{d}_{ij}\cos\angle(V,\widetilde{V})+2\sin\angle(V,\widetilde{V}).
\end{eqnarray*}
where we used $\|\widehat{\bf x}_{i}-\widehat{\bf y}_{j}\|_{2}\leq\|\widehat{\bf x}_{i}\|_2+\|\widehat{\bf y}_{j}\|_{2}\leq 2$. Using the same trick, we can prove that
$$
\widetilde{d}_{ij}\leq{d}_{ij}\cos\angle(V,\widetilde{V})+2\sin\angle(V,\widetilde{V}),
$$
a combination of the above two inequalities gives (\ref{eqn324}).
\end{proof}
\begin{rem}\label{Rem3.4}
Theorem \ref{Thm3.8} indicates that if $\sin\angle(V,\widetilde{V})$ is small {\rm(}say, $10^{-4}${\rm)}, then the
$\{d_{ij}\}'s$ and the $\{\widetilde{d}_{ij}\}'s$ will be close to each other. Thus, there is no need to compute the eigenvectors too accurately in practice. This explains why the matrix exponential eigenproblems {\rm(}\ref{26}{\rm)} and {\rm(}\ref{3.1}{\rm)} can be solved approximately in practice.

However, we cannot get the value of $\sin\angle(V,\widetilde{V})$ a priori, as the ``exact" eigenspace ${\rm span}\{V\}$ is unavailable. Let $(\lambda_1,{\bf x}_1),\ldots,(\lambda_t,{\bf x}_t)$ be the desired eigenpairs,
it was shown that if the minimal distance between the Ritz values $\widetilde{\lambda}_1,\widetilde{\lambda}_2,\ldots,\widetilde{\lambda}_t$ and the other eigenvalues {\rm(}i.e., the eigenvalues other than $\widetilde{\lambda}_1,\widetilde{\lambda}_2\ldots,\widetilde{\lambda}_t${\rm)} is sufficiently large, then $\sin\angle(V,\widetilde{V})$ is proportional to the residual norms of Ritz pairs \cite{Jia,Stewart}.
Therefore, we can use the largest residual norm of the Ritz pairs to take the place of $\sin\angle(V,\widetilde{V})$ in practice.
Experimentally, we find that a tolerance of $tol=10^{-4}$ is good enough for the matrix exponential eigenproblems. 
\end{rem}

In summary, we propose the main algorithm of this work for solving the large generalized matrix exponential eigenproblems. 


\begin{algorithm}\label{Alg3} {\bf Inexact Krylov-EDA algorithms for matrix exponential discriminant analysis}\\
Steps 1--4 are the same as those of Algorithm \ref{Alg2}.\\
5.~Given a convergence threshold $tol$ {\rm(}e.g., $tol=10^{-4}${\rm)}, compute the desired eigenvectors
by using a restarted Krylov subspace algorithm {\rm(}e.g. \cite{BCR2,Sorensen}{\rm)} for solving {\rm(}\ref{26}{\rm)} {\rm(}by using the Arnoldi method{\rm)} or {\rm(}\ref{3.2}{\rm)} {\rm(}by using the Lanczos method{\rm)}, where we use steps 6A--10A or 6L--12L in Algorithm \ref{Alg2} for
matrix exponential-vector products.\\
6.~Orthogonalize the columns of the projection matrix $V$.
\end{algorithm}

\begin{rem}
Our new methods fall within the class of the inexact Krylov subspace methods. However, the new method is different from the
inner-outer Krylov subspace methods with inexact matrix-vector products \cite{BF2,BF1,Simoncini,Simoncini2}. In those methods, the tolerance of the inner iterative process {\rm(}i.e., approximating the
matrix exponential action{\rm)} can be significantly relaxed
once the outer process {\rm(}i.e, the Arnoldi/Lanczos
eigenvalue solver{\rm)} has started to converge. However, in our new methods the matrix exponential-vector products are computed ``exactly" based on some closed-form formulae. Further, the tolerance for Arnoldi/Lanczos
eigenvalue solver in Algorithm \ref{Alg3} is determined in advance.
\end{rem}

\section{Numerical Experiments}\label{sec4}
\setcounter{equation}{0}

In this section, we make some numerical experiments to illustrate the efficiency of Algorithm 3 for face recognition.
All the numerical experiments were run
on a Dell PC with eight core Intel(R) Core(TM)i7-2600 processor with CPU
3.40 GHz and RAM 16.0 GB, under the Windows 7 with 64-bit operating system. All the numerical results were obtained from using a MATLAB 7.10.0 implementation.
In all the algorithms for comparison, the projection matrix $V$ is composed of the $k-1$ (i.e., $t=k-1$) dominant discriminant vectors, where $k$ is
the number of classes. We apply the nearest neighbor (NN) \cite{near} as the classifier with the $L_2$ metric as distance
measure.
Each column ${\bf x}_i$ of the data matrices is scaled by its 2-norm,
and 10 random splits are run so that one can obtain
a stable recognition result. In all the EDA-based algorithms, the columns of $V$ is orthonomalized by using the MATLAB built-in function {\tt orth.m}, where the {\tt orth} operation is performed via a QR decomposition, i.e., stabilized Gram-Schmidt.


For the performance of the EDA-based methods and how they compare with the current state-of-the-art, we refer to \cite{DB,WCP,Wang,YPan,ZF}. Indeed,
if the large matrix exponential eigenproblems (\ref{2.02}) and (\ref{3.2}) were solved ``exactly", then our new methods are mathematically equivalent to the original EDA method proposed in \cite{ZF}.
The main aim of this section is to show {\mbox Algorithm 3} outperforms {\mbox Algorithm 1} according to CPU time, while the classification accuracy of the former is comparable to that of the latter.

The different EDA-based algorithms are listed as follows:\\
$\bullet$ {\tt (Inexact) Arnoldi-EDA} (Algorithm 3): We apply Algorithm 3 to solve the large nonsymmetric matrix exponential eigenproblem (\ref{26}),
in which the implicitly restarted {\it Arnoldi} algorithm \cite{Sorensen} (the MATLAB built-in function {\tt eigs.m}) is used.
The matrix exponential-vector products are computed by using Algorithm \ref{Alg2}.
The stopping criterion for the large eigenproblem is chosen as $tol=10^{-4}$.\\
$\bullet$ {\tt (Inexact) Lanczos-EDA} (Algorithm 3): We apply Algorithm 3 for solving the large symmetric matrix exponential eigenproblem (\ref{3.2}),
in which the implicitly restarted {\it Lanczos} algorithm \cite{Sorensen} (the MATLAB built-in function {\tt eigs.m}) is used.
The matrix exponential-vector products are computed by using Algorithm \ref{Alg2}. The stopping criterion for the large eigenproblem is chosen as $tol=10^{-4}$.\\
$\bullet$ {\tt EDA-eigs}: Solving the generalized symmetric matrix exponential eigenproblem (\ref{2.02}) by using the MATLAB built-in function {\tt eigs.m}, the convergence threshold is chosen as $tol=10^{-8}$. This mimics forming the two matrix exponentials ${\rm exp}(S_B)$ and ${\rm exp}(S_W)$ explicitly while solving the matrix exponential eigenproblem iteratively with a relatively higher accuracy. The matrix exponentials are evaluated by using the MATLAB built-in function {\tt expm.m}. \\
$\bullet$ {\tt EDA} (Algorithm 1): The EDA algorithm advocated in \cite{ZF}, in which we form ${\rm exp}(-S_W){\rm exp}(S_B)$ and solve the nonsymmetric matrix exponential eigenproblem (\ref{26}) by using the implicit QR algorithm \cite{Bai,GV} (the MATLAB built-in function {\tt eig.m}). The matrix exponentials are evaluated by using the MATLAB built-in function {\tt expm.m}.


{\bf Example 4.1}~~In this example, we compare the EDA-based algorithms with {LDA$+$PCA} \cite{BHK} and the classical LDA method \cite{DHS,RAF}, and show the efficiency of the exponential discriminant analysis method.
A subset of AR database \cite{Mar} is used here with 1680 face images from $120$ persons (14 images per person), and all images are
cropped and scaled to
$50\times 40$. Figure 4.1 presents the sample of cropped AR database images of three individuals.
A random subset with $\ell=2,3,5$ images per subject is
taken to form the training set, and the rest of the images are
used as the testing set. We run {Arnoldi-EDA}, {Lanczos-EDA}, {EDA-eigs}, {EDA} and LDA+PCA \cite{BHK} on this problem.
As a by-product, we also list the numerical results of
the ``Classical LDA" method \cite{DHS,RAF}, i.e.,
the classical LDA method with
the QZ algorithm \cite{GV} for the generalized eigenproblem (\ref{1.3}) (by using the MATLAB built-in function {\tt eig.m}).
As was done in \cite{ZF}, for LDA$+$PCA, we reserve $99\%$ energy in the PCA stage, followed by LDA.
Table 4.1 lists the CPU time (in seconds) and the recognition accuracy obtained from the different methods. When $\ell=3$, we plot in Figure 4.2 recognition accuracy of the six algorithms with respect to the projected dimension on the AR database. Notice that the curves of {Arnoldi-EDA}, {EDA-eigs}, and {EDA} overlap with each other.


\begin{figure}
{\small {\bf Figure 4.1,~Example 4.1}:~Sample face images of 3 individuals from
the AR database, $d=50\times 40$.}
\end{figure}

\begin{figure}
\begin{center}
\scalebox{0.5}{\includegraphics{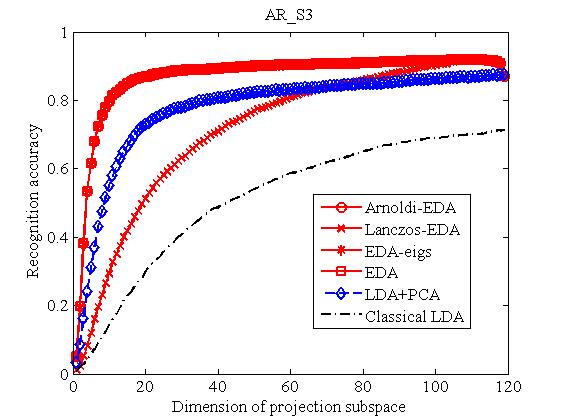}}\\
\end{center}
{\small {\bf Figure 4.2, Example 4.1}:~Recognition accuracy with respect to the projected dimension on the AR database, 3 Train.}
\end{figure}

{\small
\begin{table}[!h]
\begin{center}
\def\temptablewidth{0.7\textwidth}
{\rule{\temptablewidth}{0.8pt}}
\begin{tabular*}{\temptablewidth}{@{\extracolsep{\fill}}lccccr}
{\bf Algorithm}  &{\bf 2 Train} &{\bf 3 Train} & {\bf 5 Train}    \\\hline
Arnoldi-EDA & 0.51(84.3\%) &0.71(92.0\%) &1.28 (97.0\%)   \\
Lanczos-EDA & 0.55(84.6\%) &0.80(92.2\%) &1.58 (97.3\%)   \\
EDA-eigs    &6.15(84.3\%) &6.69(92.0\%) &7.35 (97.0\%)   \\
EDA     &9.61(84.3\%) &9.75(92.0\%) &10.0 (97.0\%)   \\\hline
LDA+PCA     &0.09(82.6\%) &0.21(87.7\%) &0.46 (94.8\%)\\
Classical LDA   &51.4(65.0\%) &52.1(71.3\%) &52.0 (78.0\%)   \\
 \end{tabular*}
 {\rule{\temptablewidth}{1pt}}\\
 \end{center}
 \begin{flushleft}
  {\small {\rm {\bf Table 4.1, Example 4.1}:~CPU time in seconds and recognition accuracy (in brackets) of the six algorithms on the AR database, $d=50\times 40,~k=120$. }}
 \end{flushleft}
 \end{table}
}

We see from Table 4.1 that the recognition accuracy obtained from the four EDA-based algorithms are about the same, which are (much) higher
than those from LDA+PCA and classical LDA. Furthermore, both {Arnoldi-EDA} and {Lanczos-EDA} converge much faster than {EDA-eigs} and {EDA}, while {EDA-eigs} outperforms {EDA}.
The classical LDA method works the poorest in terms of CPU time and recognition accuracy. This is due to the fact that the dimension $d=2000>n=1680$, and we suffer from the SSS problem. Recall that the classical LDA method cannot cure this drawback properly.

{\bf Example 4.2}~~The aim of this example is two-fold. First, we show that our new algorithms {Arnoldi-EDA} and {Lanczos-EDA} run much faster than EDA
and EDA-eigs for face recognition. Second, we illustrate the effectiveness of Theorem \ref{Thm3.7}. There are two data sets in this example.
The first one is the Yale face database taken from
the Yale Center for Computational Vision and Control.
It contains $165$ grayscale images of $k=15$ individuals. The images
demonstrate variation with the following expressions or configurations:
1) lighting: center light, left light, and right light;
2) with/without glasses; and 3) facial expressions: normal,
happy, sad, sleepy, surprised, and winking. The original image
size is $320\times 243$ pixels. Figure 4.3 gives the sample of cropped Yale database images of three individuals.

The second test set is the Extended YaleB database. 
This database contains 5760 single light source images of 10 subjects each seen under 576 viewing conditions (9 different poses and 64 illumination conditions of each person). The images have normal, sleepy, sad and surprising expressions. There are some images with or without glasses. These images are captured by varying the position of light source at the center, left or right. For every subject in a particular pose, an image with ambient (background) illumination was also captured.
A subset of $k=38$ with 2432 images are used in this example (64 images of per individual with illumination). Figure 4.4 shows the sample of cropped Extended YaleB database images of three individuals.

\begin{figure}
{\small {\bf Figure 4.3,~Example 4.2}:~Sample face images of 3 individuals from
the Yale database, $d=64\times 64$.}
\end{figure}

\begin{figure}
{\small {\bf Figure 4.4,~Example 4.2}:~Sample face images of 3 individuals from
the Extended YaleB database, $d=64\times 64$.}
\end{figure}

In this example, all images are
aligned based on eye coordinates and are cropped and scaled to
$32\times 32,64\times 64$ and $100\times 100$, respectively.
A random subset with $\ell=3,5,8$ images per subject is
taken to form the training set, and the rest of the images are
used as the testing set. Tables 4.2 and 4.3 report the CPU time (in seconds) and the recognition accuracy of the different methods.

\begin{figure}
\begin{center}
\scalebox{0.5}{\includegraphics{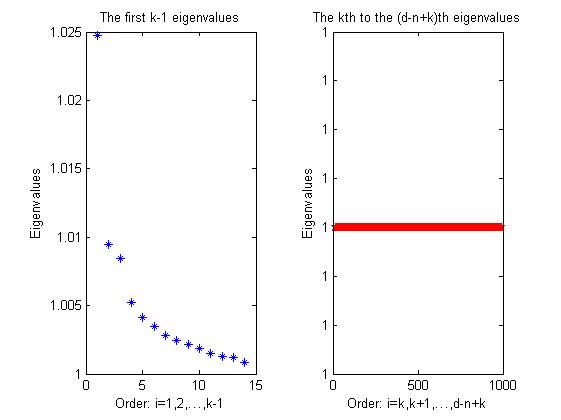}}\\
\end{center}
{\small {\bf Figure 4.5, Example 4.2}:~The first $k-1$ eigenvalues (Left) and the $k$-th to the $(d-n+k)$-th eigenvalues (Right) of ${\rm exp}(-S_W){\rm exp}(S_B)$, Yale data base,~$d=32\times 32,~k=15$; 3 Train.}
\end{figure}

Three remarks are in order. First, our new algorithms {Arnoldi-EDA} and {Lanczos-EDA} outperform EDA-eigs and EDA considerably in terms of CPU time, while the recognition accuracy of the four algorithms are about the same.
For example, we see from Table 4.2 that, when $d=100\times 100$ and $\ell=3$, {Arnoldi-EDA} and {Lanczos-EDA} used 0.12 and 0.11 seconds, while {EDA-eigs} and {EDA} used 409.4 and 1027.4 seconds, respectively. Thus, both the {Arnoldi-EDA} and the {Lanczos-EDA} algorithms can circumvent the drawback of heavily computational complexity that bothers the original EDA algorithm, while keep comparable recognition accuracy. Second, cropping the original images may lose some useful information and thus may result in a low recognition accuracy. For instance, for the Yale database, when $\ell=5$, the recognition accuracy of {Arnoldi-EDA} applying to $d=100\times 100$ images is $93.8\%$, while the recognition accuracy is only $74.7\%$ as $d=32\times 32$. Third, it is seen from Tables 4.2--4.3 that we cannot tell which one, {Arnoldi-EDA} or {Lanczos-EDA}, is {\it definitely better} than the other in terms of CPU time and recognition accuracy.

In order to show effectiveness of Theorem \ref{Thm3.7}, we plot in Figure 4.5
the first $k-1$ eigenvalues and the $k$-th to the $(d-n+k)$-th eigenvalues of the matrix ${\rm exp}(-S_W){\rm exp}(S_B)$ as $d=32\times 32$.
One observes that the 4-th to the $(k-1)$-th eigenvalues are clustered, and there are at least $d-n+1$ eigenvalues equal to 1. All these coincide with the theoretical results given in
Theorem \ref{Thm3.7}.

{\small
\begin{table}[!h]
\begin{center}
\def\temptablewidth{0.9\textwidth}
{\rule{\temptablewidth}{0.8pt}}
\begin{tabular*}{\temptablewidth}{@{\extracolsep{\fill}}lccccr}
{\bf Algorithm}   &$d$  &{\bf 3 Train}       &{\bf 5 Train} &{\bf 8 Train}    \\\hline
Arnoldi-EDA  &$32\times 32$  &0.04(65.3\%) &0.04(74.7\%) &0.06(83.1\%)   \\
Lanczos-EDA  &$32\times 32$  &0.02(65.2\%) &0.02(73.8\%) &0.04(83.1\%)   \\
EDA-eigs     &$32\times 32$  &0.76(65.8\%) &0.79(74.7\%) &0.78(83.1\%)   \\
EDA      &$32\times 32$  &1.24(65.8\%) &1.23(74.7\%) &1.24(83.1\%)   \\\hline
Arnoldi-EDA  &$64\times 64$  &0.06(74.0\%) &0.08(83.8\%) &0.11(90.2\%) \\
Lanczos-EDA  &$64\times 64$  &0.04(72.8\%) &0.07(84.9\%) &0.11(90.2\%)   \\
EDA-eigs     &$64\times 64$  &29.7(74.8\%) &31.1(84.7\%) &31.4(89.8\%)   \\
EDA      &$64\times 64$  &73.0(74.8\%) &74.7(84.7\%) &75.1(89.8\%)   \\\hline
Arnoldi-EDA  &$100\times 100$  &0.12(88.3\%) &0.15(93.8\%) &0.25(96.4\%)\\
Lanczos-EDA &$100\times 100$  &0.11(88.5\%) &0.16(95.3\%) &0.31(96.0\%)   \\
EDA-eigs     &$100\times 100$  &409.4(88.0\%) &397.1(94.8\%) &401.4(96.9\%)   \\
EDA      &$100\times 100$  &1027.4(88.0\%) &993.5(94.9\%) &1009.2(96.9\%)   \\
 \end{tabular*}
 {\rule{\temptablewidth}{1pt}}\\
 \end{center}
 \begin{flushleft}
  {\small {\rm {\bf Table 4.2, Example 4.2}:~CPU time in seconds and recognition accuracy (in brackets) of the four algorithms on the Yale database, $k=15$.}}
 \end{flushleft}
 \end{table}
}

{\small
\begin{table}[!h]
\begin{center}
\def\temptablewidth{0.9\textwidth}
{\rule{\temptablewidth}{0.8pt}}
\begin{tabular*}{\temptablewidth}{@{\extracolsep{\fill}}lccccr}
{\bf Algorithm}   &$d$  &{\bf 3 Train}  &{\bf 5 Train} & {\bf 8 Train}    \\\hline
Arnoldi-EDA  & $32\times 32$  & 0.07(55.9\%) &0.11(73.0\%) & 0.16(82.8\%)   \\
Lanczos-EDA  & $32\times 32$  & 0.06(59.5\%) &0.11(71.2\%) & 0.15(82.9\%)   \\
EDA-eigs     & $32\times 32$  & 1.02(56.6\%) &1.16(73.2\%) & 1.28(83.7\%)   \\
EDA      & $32\times 32$  & 1.28(56.6\%) &1.34(73.2\%) & 1.35(83.7\%)   \\\hline
Arnoldi-EDA  & $64\times 64$  & 0.23(56.3\%) &0.50(72.9\%) & 0.78(83.2\%) \\
Lanczos-EDA  & $64\times 64$  & 0.23(59.4\%) &0.58(71.1\%) & 0.99(83.0\%)   \\
EDA-eigs     & $64\times 64$  & 36.1(56.9\%) &37.7(73.4\%) & 41.2(84.1\%)   \\
EDA      & $64\times 64$  & 78.0(56.9\%) &78.4(73.4\%) & 81.4(84.1\%)   \\\hline
Arnoldi-EDA  & $100\times 100$  & 0.60(56.6\%) &1.46(73.9\%) & 1.90(83.2\%)\\
Lanczos-EDA & $100\times 100$  & 0.74(59.4\%) &1.88(71.4\%) &2.46(83.3\%)   \\
EDA-eigs     & $100\times 100$  & 421.7(57.3\%) &444.8(73.7\%) &541.8(84.2\%)   \\
EDA      & $100\times 100$  & 1013.2(57.3\%) &1030.0(73.7\%) &1129.7(84.2\%)   \\
 \end{tabular*}
 {\rule{\temptablewidth}{1pt}}\\
 \end{center}
 \begin{flushleft}
  {\small {\rm {\bf Table 4.3, Example 4.2}:~CPU time in seconds and recognition accuracy (in brackets) of the four algorithms on the Extended YaleB database, $k=38$.}}
 \end{flushleft}
 \end{table}
}

{\bf Example 4.3}~~In this example, we show efficiency of our ``inexact" strategy (see Theorem \ref{Thm3.8} and Remark \ref{Rem3.4}) for solving the large matrix exponential eigenproblems (\ref{26}) and (\ref{3.2}).
The FERET database is one of the standard image database specially used for the face recognition algorithms \cite{FERET}. The final corpus consists of 14051 eight-bit grayscale images of human heads with views ranging from frontal to left and right profiles. In this example, we consider a subset of 1400 images of $k=200$ individuals, in which each person contributing seven images. The seven images of each individual consists of different illumination and expression variation.
Before experiment, the facial images are cropped to a size of $80\times 80$ pixels. Figure 4.6 lists the sample of cropped FERET database images of three individuals. A random subset with $\ell=2,3,5$ images per subject is
taken to form the training set, and the rest of the images are
used as the testing set. In this example, we have to compute $k-1=199$ dominant eigenpairs of a $6400\times 6400$ matrix using {Arnoldi-EDA} and {Lanczos-EDA}, which is a challenging task.

In order to show efficiency of our inexact strategy, we run the (implicitly restarted) Arnoldi and Lanczos algorithm in Step 5 of Algorithm 3 with convergence tolerance $tol=10^{-2},10^{-4},10^{-6},10^{-8}$ and $10^{-10}$, respectively. Table 4.4 presents the numerical results.
We observe that the two inexact algorithms work quite well in all the cases, and a tolerance of ${\rm tol}=10^{-4}$ is good enough. Indeed, as is shown by Theorem \ref{Thm3.8}, it is unnecessary to solve the large exponential eigenproblems in a very high accuracy.
This is favorable to very large matrix computations arising from high dimensionality reduction.

It is seen that the recognition rates of {Lanczos-EDA} are lower than those of {Arnoldi-EDA} for this problem. A possible reason is that the approximation {\bf y} from {Lanczos-EDA} undergoes a exponential transformation to get the desired solution {\bf x}, see (\ref{eqn311}). As a comparison, we also list the numerical results obtained from running EDA and LDA+PCA. One observes that our new algorithms run much faster than EDA, and are superior to LDA+PCA in terms of recognition accuracy.

\begin{figure}
\begin{flushleft}
{\small {\bf Figure 4.6, Example 4.3 }:~Sample face images of 3 individuals from
the FERET database, $d=80\times 80$.}
\end{flushleft}
\end{figure}

{\small
\begin{table}[!h]
\begin{center}
\def\temptablewidth{0.9\textwidth}
{\rule{\temptablewidth}{0.8pt}}
\begin{tabular*}{\temptablewidth}{@{\extracolsep{\fill}}lccccr}
{\bf Algorithm}   &$tol$  &{\bf 3 Train}       &{\bf 5 Train} & {\bf 8 Train}    \\\hline
Arnoldi-EDA & $10^{-2}$  & 5.22(62.3\%) &6.69(73.4\%) & 10.2 (81.2\%)   \\
             & $10^{-4}$  & 5.22(62.3\%) &6.68(73.3\%) & 10.3 (81.2\%)   \\
             & $10^{-6}$  & 5.22(62.3\%) &10.4(73.4\%) & 21.6 (81.2\%)   \\
             & $10^{-8}$  & 5.27(62.4\%) &11.4(73.4\%) & 24.6 (81.3\%)   \\
             & $10^{-10}$  & 5.27(62.3\%) &12.0(73.4\%) &26.8 (81.3\%)   \\\hline
Lanczos-EDA  & $10^{-2}$  & 6.41(44.2\%) &8.64(61.7\%) &13.8 (79.1\%)   \\
             & $10^{-4}$  & 6.41(44.2\%) &8.64(61.7\%) &13.8 (79.1\%)   \\
             & $10^{-6}$  & 6.42(44.2\%) &14.5(61.6\%) &28.0 (79.2\%)   \\
             & $10^{-8}$  & 6.49(44.2\%) &15.6(61.6\%) &34.9 (79.2\%)  \\
             & $10^{-10}$  & 6.51(44.2\%) &16.5(61.6\%) &38.0 (79.2\%)   \\\hline
EDA          &--           &288.6(62.3\%) &277.7(73.5\%) &286.5 (81.4\%)   \\
LDA$+$PCA    &--          &0.56(42.3\%) &1.17(39.7\%) & 2.54 (57.5\%) \\
 \end{tabular*}
 {\rule{\temptablewidth}{1pt}}\\
 \end{center}
 \begin{flushleft}
  {\small {\rm {\bf Table 4.4, Example 4.3 }:~CPU time in seconds and recognition accuracy (in brackets) of the Arnoldi-EDA and Lanczos-EDA algorithms (with different tolerances for exponential eigenproblems) on the FERET database, $d=80\times 80,~k=200$.}}
 \end{flushleft}
 \end{table}
}

{\bf Example 4.4}~~In this example, we demonstrate that {Arnoldi-EDA} and {Lanczos-EDA} are suitable to data sets with high dimension.
There are two test sets in this example.
The ORL database contains 400 face images of $k=40$
individuals, and the image size is $92\times 112$. The major challenge on this data set is the
variation of the face pose. There is no lighting variation with
minimal facial expression variations and no occlusion. Figure 4.7 gives the sample of ORL database images of three individuals. A random subset with $\ell=2,3,5$ images per subject was
taken to form the training set, and the rest of the images were
used as the testing set.

The CMU-PIE database contains more than 40,000 images of 68 subjects with more than 500 images for each. These face images are
captured by 13 synchronized cameras and 21 flashes under varying pose, illumination, expression and lights. In this experiment, we choose $k=10$ subjects and 10 images under different illuminations, lights, expressions and poses for each subject. Thus, the total number of images chosen from CMU-PIE database is 100. The image size is $d=486\times 640$ pixels, so we have to deal with eigenproblems of size $311,040\times 311,040$. We randomly select
with $\ell=3,5,8$ images for each subject to organize a training set, while the testing set consists of the remaining images. Figure 4.8 depicts the sample of CUM-PIE database images of three individuals.

In this example,
we compare the inexact Krylov subspace algorithms {Arnoldi-EDA} and {Lanczos-EDA}
with some state-of-the-art algorithms including EDA\cite{ZF}, PCA \cite{PCA}, LDA$+$PCA (Fisherfaces) \cite{BHK},
NLDA (the null space LDA method) \cite{CLKLY}, LDA/QR \cite{YL}, LDA/GSVD \cite{HJP,YJPP}, RLDA \cite{JHF} (the regularized LDA method, where the regularized parameter is chosen as 0.01) for the face recognition problem. We reserve $99\%$ energy for recognition for PCA.
For LDA$+$PCA, we also reserve $99\%$ energy in the PCA stage, followed by LDA.
Tables 4.5 and 4.6 list the numerical results. In Figures 4.9 and 4.10, we plot curves of recognition accuracy with respect to the projected dimension on the ORL and CUM-PIE databases, with $\ell=5$ and 3, respectively.

\begin{figure}
{\small {\bf Figure 4.7,~Example 4.4}:~Sample face images of 3 individuals from
the ORL database, $d=92\times 112$.}
\end{figure}

\begin{figure}
{\small {\bf Figure 4.8,~Example 4.4}:~Sample face images of 3 individuals from
the CMU-PIE database, $d=640\times 486$.}
\end{figure}

It is observed from Table 4.5 and Figure 4.9 that the recognition accuracies of the four EDA-based algorithms are higher than those of many state-of-the-art algorithms such as LDA+PCA, PCA, LDA/QR and LDA/GSVD.
This coincides with the numerical results reported in \cite{DB,WCP,Wang,YPan,ZF}. We notice that the recognition accuracies of RLDA and NLDA are comparable to
those of the four EDA-based methods, while {Arnoldi-EDA} and {Lanczos-EDA} run much faster than RLDA, NLDA and LDA/GSVD.

Moreover, it is seen from Table 4.6 and Figure 4.10 that for very large data sets such as CMU-PIE, all the algorithms EDA-eigs, EDA, RLDA, NLDA and LDA/GSVD fail to work, due to their heavy storage requirements.
As a comparison, {Arnoldi-EDA} and {Lanczos-EDA} perform quite well in all the situations, and their recognition accuracies are a little higher than those of PCA,
LDA+PCA and LDA/QR. Consequently, our new algorithms are suitable to data sets with high dimension, which are competitive alternatives to EDA for dimensionality reduction.

\begin{figure}
\begin{center}
\scalebox{0.5}{\includegraphics{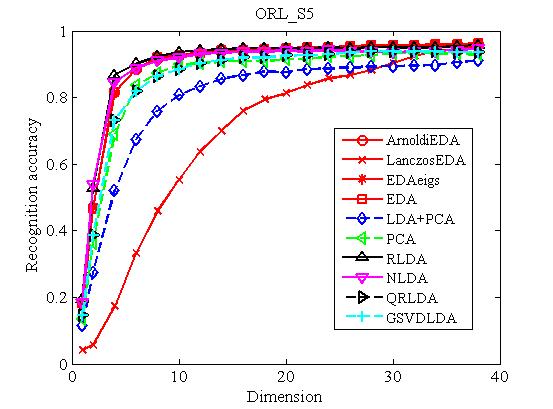}}
\end{center}
 {\small {\bf Figure 4.9, Example 4.4}:~Recognition accuracy with respect to projected dimension on the ORL database, 5 Train.}
\end{figure}

\begin{figure}
\begin{center}
\scalebox{0.5}{\includegraphics{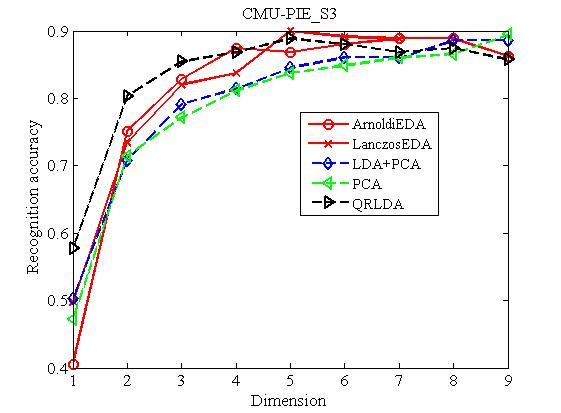}}
\end{center}
 {\small {\bf Figure 4.10, Example 4.4}:~Recognition accuracy with respect to projected dimension on the CMU-PIE database, 3 Train.}
\end{figure}

%

{\small
\begin{table}[!h]
\begin{center}
\def\temptablewidth{0.8\textwidth}
{\rule{\temptablewidth}{0.8pt}}
\begin{tabular*}{\temptablewidth}{@{\extracolsep{\fill}}lccccr}
{\bf Algorithm}  &{\bf 2 Train} &{\bf 3 Train} & {\bf 5 Train}    \\\hline
Arnoldi-EDA & 0.36(84.3\%) &0.51(87.8\%) &0.72 (96.4\%)   \\
Lanczos-EDA & 0.41(83.1\%) &0.63(86.9\%) &0.98 (95.7\%)   \\
EDA-eigs    & 407.9(84.3\%) &418.2(87.8\%) &426.0 (96.4\%)   \\
EDA     &1079.9(84.3\%) &1090.9(87.8\%) &1098.9 (96.4\%)   \\\hline
LDA+PCA   &0.05(83.2\%) &0.08(86.7\%) & 0.17 (91.6\%) \\
PCA         &0.03(78.4\%) &0.06(83.4\%) & 0.12 (93.5\%)   \\
RLDA        &1142.3(84.3\%)  &1165.7(88.3\%) &1159.0 (95.8\%)   \\
NLDA        &12.0(83.8\%) &17.4(87.7\%) & 16.5 (94.8\%)   \\
LDA/QR      &0.09(80.2\%) &0.12(85.0\%) & 0.17 (94.0\%)   \\
LDA/GSVD    &10.9(82.3\%) &16.3(85.9\%) & 15.5 (93.6\%)   \\
 \end{tabular*}
 {\rule{\temptablewidth}{1pt}}\\
 \end{center}
 \begin{flushleft}
  {\small {\rm {\bf Table 4.5, Example 4.4}:~CPU time in seconds and recognition accuracy (in brackets) of the 10 algorithms on the ORL database, $d=92\times 112 ,~k=40$. }}
 \end{flushleft}
 \end{table}
}

{\small
\begin{table}[!h]
\begin{center}
\def\temptablewidth{0.7\textwidth}
{\rule{\temptablewidth}{0.8pt}}
\begin{tabular*}{\temptablewidth}{@{\extracolsep{\fill}}lccccr}
{\bf Algorithm}  &{\bf 3 Train} &{\bf 5 Train} & {\bf 8 Train}    \\\hline
Arnoldi-EDA & 1.98(88.9\%) &3.44(94.4\%) &5.81(92.0\%)   \\
Lanczos-EDA & 2.46(90.0\%) &4.17(92.4\%) &6.93(91.0\%)   \\
EDA-eigs    &-- &-- &--   \\
EDA      & -- &-- &--    \\\hline
LDA+PCA   &0.36(88.6\%) &0.74(94.0\%) & 1.15(92.0\%) \\
PCA         &0.31(82.3\%) &0.67(88.4\%) & 1.06(89.0\%)   \\
RLDA         &-- &-- &--   \\
NLDA        &-- &-- &--   \\
LDA/QR      &0.47(88.9\%) &0.64(93.6\%) & 0.89(92.0\%)   \\
LDA/GSVD     &-- &-- &--   \\
 \end{tabular*}
 {\rule{\temptablewidth}{1pt}}\\
 \end{center}
 \begin{flushleft}
  {\small {\rm {\bf Table 4.6, Example 4.4}:~CPU time in seconds and recognition accuracy (in brackets) of the 10 algorithms on the CMU-PIE database, $d=640\times 486,~k=10$. Here ``--" means the algorithm suffers from ``Out of Memory".}}
 \end{flushleft}
 \end{table}
}


\section{Conclusions and future work}
The computation of large scale matrix exponential eigenproblem is the bottleneck in the frame work of EDA for high
dimensionality reduction \cite{Ah,DB,WCP,Wang,YPan,ZF}.
In this paper, we propose two inexact Krylov subspace algorithms, i.e., the inexact Arnoldi algorithm and the inexact Lanczos algorithm for nonsymmetric and symmetric matrix exponential eigenproblems, respectively.
Our main contribution is to investigate computing matrix exponential-vector products efficiently, and to solve the matrix exponential eigenproblems ``approximately".
The relationship between the accuracy of the approximate eigenvectors and the distance to the nearest neighbour classifier (NN) is revealed. A theoretical comparison of the LDA criterion and the EDA criterion is also given.
Experimental results on popular
databases for face recognition have validated the effectiveness of the proposed inexact Krylov subspace algorithms.

We point out that our new strategies also apply to other exponential based methods such as the exponential Laplacian embedding, exponential LDE, exponential LPP, exponential discriminant regularization, and 2DEDA for high dimensionality reduction other than face recognition \cite{Ah,DB,WCP,Wang,YPan,ZF}.
Furthermore, how to combine our new strategies with other eigen-solvers deserves further investigation.
For instance, the Jacobi-Davidson method \cite{JD} is a
successful eigenvalue solver for large eigenproblems. This method is applicable to the matrix exponential eigenproblems (\ref{26}) and (\ref{3.1}).
Furthermore, the bound {\rm(}\ref{eqn324}{\rm)} only relates the perturbation of the projections to the perturbation of the projected points for the nearest neighbor {\rm(}NN{\rm)}. It is interesting to establish a performance bound for general KNN which should also involve the distances between points of different classes.

\end{document}